\documentclass{lmcs} 
\pdfoutput=1

\usepackage{lastpage}
\lmcsdoi{15}{2}{21}
\lmcsheading{}{\pageref{LastPage}}{}{}%
{Feb.~20,~2018}{Jun.~25,~2019}{}


\usepackage{hyperref}
\usepackage[british]{babel}
\usepackage[all]{xy}
\SelectTips{cm}{}

\usepackage{amsthm,amssymb,mathrsfs,xspace}
\def\nameit#1{#1~}
\def\thx{\nameit{Theorem}}
\def\lmx{\nameit{Lemma}}
\def\prx{\nameit{Proposition}}
\def\crx{\nameit{Corollary}}
\def\dfx{\nameit{Definition}}
\def\rmx{\nameit{Remark}}
\def\exx{\nameit{Example}}

\mathcode`\:="603A 
\mathchardef\colon="303A 

\DeclareFontFamily{OT1}{pzc}{}
\DeclareFontShape{OT1}{pzc}{m}{it}{<->s*[1.14]pzcmi7t}{}
\DeclareMathAlphabet{\mathpzc}{OT1}{pzc}{m}{it}

\def\thmitem{\def\labelenumi{{\normalfont(\roman{enumi})}}}
\def\dfn#1{{\bfseries\itshape #1\/}}

\def\des#1{\ensuremath{\ct{D}es_{#1}}\xspace}
\def\id#1{\ensuremath{\mathrm{id}_{#1}}}
\def\ev{\ensuremath{\mathrm{ev}}}
\def\Id#1{\ensuremath{\mathrm{Id}_{#1}}}
\def\op{^{\textrm{\scriptsize op}}}
\def\opp{\strut^{\textrm{\tiny op}}}

\def\ct#1{\ensuremath{\mathpzc{#1}}\xspace}
\def\Ct#1{\ensuremath{{\normalfont\textsf{\bfseries #1}}}\xspace}
\def\spe#1{\ensuremath{\mathbf{#1}}}
\def\rst{\mathord{\restriction}}
\def\Kl#1{\ensuremath{\mathbb{K}_{#1}}}

\def\RB#1{\mathchoice
{\rotatebox[origin=c]{180}{$#1$}}
{\rotatebox[origin=c]{180}{$#1$}}
{\rotatebox[origin=c]{180}{$\scriptstyle#1$}}
{\rotatebox[origin=c]{180}{$\scriptscriptstyle#1$}}}
\def\D{\RB{E}\kern-.3ex}
\def\B{\RB{A}\kern-.6ex}

\def\exl{_{\textrm{\normalfont\scriptsize ex/lex}}}
\def\rgl{_{\textrm{\normalfont\scriptsize reg/lex}}}
\def\cmp#1{\ensuremath{\{\kern-2.5pt|{#1}|\kern-2.5pt\}}}
\def\Set{\ct{Set}}
\let\Iff\Leftrightarrow
\def\spsymb#1{\ensuremath{\mathsf{#1}}\xspace}
\def\NN{\spsymb{N}}
\def\N{\ensuremath{\mathbb{N}}\xspace}
\def\oo{\spsymb{0}}
\def\ss{\spsymb{s}}

\def\TU{T}
\def\UU{U}
\def\EH{\Ct{ED}}
\def\ISL{\Ct{InfSL}}
\def\Hey{\Ct{Heyt}}

\def\eff{\ct{Eff}}
\def\ass{\ct{Asm}}
\def\pass{\ct{PAsm}}

\def\ha{{\bf HA}\xspace}

\def\stm{\ensuremath{\pw\Gamma}\xspace}
\def\mf{{\bf MF}\xspace}
\def\vuoto{}
\def\LL#1{\ensuremath{\relax\def\prova{#1}\relax
\D_{\ifx\prova\vuoto\blank\else#1\fi}(\tt)}}
\def\RR{\ec{\cmp{\blank}}}

\def\Q#1{\ensuremath{\widehat{#1}}}
\def\ec#1{\ensuremath{\left[{#1}\right]}}
\def\blank{\mathrm{-}}
\def\ple#1{\ensuremath{\langle #1\rangle}}
\def\epl#1{\ensuremath{\langle\kern-.55ex\langle #1\rangle\kern-.55ex\rangle}}
\def\tur#1{\ensuremath{\varphi_{#1}}}
\def\fs#1{\ensuremath{{#1_0}}}
\def\sn#1{\ensuremath{{#1_1}}}
\def\rz#1{\ensuremath{R_{#1}}}
\def\rkl{\textrm{\scriptsize Kleene}}
\def\dsi#1{\ensuremath{{#1}^{\vphantom{I}\smash{\Vdash}}}}
\def\gm#1{\ensuremath{\gamma_{#1}}}
\def\A#1{\ensuremath{{\mathcal R}_{{#1}}}}
\def\fp#1{\ensuremath{P_{#1}}}
\def\pr{\mathrm{pr}}
\def\fst{{\pr_1}}\def\snd{{\pr_2}}
\def\pw{\ensuremath{\mathbb{P}}}

\def\qu(#1){\overline{#1}}
\def\eq#1{\mathrel{\mathord=_{#1}}}
\let\Land\wedge

\let\Implies\Rightarrow
\let\tt\top
\let\ff\bot
\def\Forall#1.{\forall_{#1\ }}
\def\Exists#1.{\exists_{#1\ }}
\def\Existu#1.{\exists!_{#1\ }}
\def\PI#1.{\Pi_{#1\ }}
\def\SIGMA#1.{\Sigma_{#1\ }}
\def\vdashv{\dashv\vdash}

\def\larr{\mathrel{\xymatrix@1@=3.5ex{*=0{}\ar[];[r]&*=0{}}}}
\def\to{\mathrel{\xymatrix@1@=2.5ex{*=0{}\ar[];[r]&*=0{}}}}
\def\ntr{\mathrel{\xymatrix@1@=2.5ex{*=0{}
\ar[];[r]|{\raisebox{.9ex}{\makebox[0pt]{$\cdot$}}}&*=0{}}}}

\let\ftr\larr
\def\Sub{\mathop{{}\mathrm{Sub}}\nolimits}
\def\SSub{\mathop{{}\mathrm{Sts}}\nolimits}
\def\Sb#1{\ensuremath{\Sub_{#1}}}
\def\Ssb#1{\ensuremath{\SSub_{#1}}}
\def\Wsb#1{\ensuremath{\Psi\kern-.4ex_{#1}}}

\def\wsd{variational doctrine\xspace}
\def\la#1.#2{\ensuremath{\lambda#1.#2}}
\def\hyper{hyperdoctrine\xspace}

\def\whyper{weak hyperdoctrine\xspace}
\def\whypers{weak hyperdoctrines\xspace}

\def\nno{\mbox{pnno}\xspace}
\def\wnno{\mbox{wpnno}\xspace}
\def\skolem{Skolem\xspace}
\def\ska{\skolem arrow\xspace}
\def\aritm{arithmetic\xspace}
\def\Aritm{Arithmetic\xspace}

\def\nfbf#1{\mbox{\normalfont\textsf{#1}}\xspace}
\def\ttch{\nfbf{Rec}}
\def\kleene#1{\ensuremath{\nfbf{K}_{#1}}}
\def\wttct{\nfbf{TCT}}
\def\ttct{\nfbf{TCT}}
\def\fct{\nfbf{CT}}

\def\ruc{{\normalfont(RUC)}\xspace}
\def\ac{{\normalfont(AC)}\xspace}
\def\auc{{\normalfont(AUC)}\xspace}
\def\rc{{\normalfont(RC)}\xspace}

\def\rcn{{\normalfont(AC$_{\NN}$)}\xspace}

\def\ie{{\em i.e.}\xspace}
\def\eg{{\em e.g.}\xspace}
\def\loccit{{\em loc.cit.}\xspace}

\begin{document}

\title[Element.\ Quotient Completions,
Church's Thesis, Part.\ Assemblies]
{Elementary Quotient Completions,\\
Church's Thesis, and Partitioned Assemblies}

\author[M.~E.~Maietti]{Maria Emilia Maietti\rsuper{a}}
\author[F.~Pasquali]{Fabio Pasquali\rsuper{b}}

\address{\lsuper{a,b}Dipartimento di Matematica ``Tullio Levi Civita'',
Universit\`a di Padova,
via Trieste 63, 35121 Padova, Italy,}
\email{maietti@math.unipd.it}
\email{pasquali@dima.unige.it}

\author[G.~Rosolini]{Giuseppe Rosolini\rsuper{c}}
\address{\lsuper{c}DIMA, Universit\`a di Genova,
via Dodecaneso 35, 16146 Genova, Italy}
\email{rosolini@unige.it}
\thanks{Projects EU-MSCA-RISE project 731143 "Computing with Infinite
Data" (CID), MIUR-PRIN 2010-2011 and Correctness by Construction (EU
7th framework programme, grant no.~PIRSES-GA-2013-612638) provided
support for the research presented in the paper.}




\begin{abstract}\noindent
Hyland's effective topos offers an important realizability model
for constructive mathematics in the form of a category whose internal
logic validates Church's Thesis. It also contains a
boolean full sub-quasitopos of ``assemblies'' where only a
restricted form of Church's Thesis survives.
In the present paper we compare the effective topos and the quasitopos
of assemblies each as the elementary quotient completions of a Lawvere
doctrine based on the partitioned assemblies. In that way we can
explain why the two forms of Church's Thesis each category satisfies
differ by the way each is inherited from specific properties of the
doctrine which determines the elementary quotient completion.
\end{abstract}

\maketitle

\section*{Introduction}

Hyland's paper ``The Effective Topos'' \cite{eff}, introducing and
studying the category \eff in the title of the paper, opened a
new way to apply techniques developed in realizability to analyse
extensively various aspects of constructive mathematics and of computer 
science, combining them with the essential use of category theory, see 
\cite{oostenstory,RosoliniG:disoet,HylandJ:algtpm,RosoliniG:extp,FreydP:funp}. 
The effective topos is the first example of an elementary
non-Grothendieck topos with a natural number object. It also provides
a computational interpretation of the logic of a topos, see
\cite{BoileauA:logt}, and \cite{MaiettiM:modcbd} for a dependent
type-theoretic version of it. Indeed the interpretation of the
internal logic in \eff extends Kleene's realizability interpretation
of Intuitionistic Arithmetic \cite{KleeneS:intint}, validates formal
Church's Thesis \fct, and the statement that every Cauchy real
is computable, see \cite{eff}.

In \loccit, the full subcategory \ass on the $\neg\neg$-separated
objects of \eff is also introduced and studied---those objects have
later been christened ``assemblies'', hence the shorthand \ass for the
full subcategory they determine, see
\cite{CarboniA:catarp,OostenJ:reaait}. In the category \ass the
endoarrows on the natural number object correspond exactly to the
computable functions. But in the (boolean) logic of the strong
subobjects of the quasitopos \ass, not all Cauchy reals defined as
(equivalence classes of certain) functional relations are computable
and \fct does not hold. In that logic, only a restricted form of
\fct---expressing internally that the arrows on the natural numbers in
\ass are computable---survives and it is called Type-Theoretic
Church's Thesis, in short $\wttct$, see \dfx\ref{churchThesis} for the
precise forms of these principles. That in turn implies that the Axiom
of Unique Choice, even on the natural numbers, does not hold in the
logic of strong subobjects in \ass. Instead the Axiom of Unique
Choice, and even the Axiom of Countable Choice and \fct, hold in
the internal logic of subobjects of \ass, see \rmx\ref{remq}.

In this paper we show that each of the categories \eff and
\ass can be viewed as the domain $\ct{Q}_P$ of the ``elementary
quotient completion'' $$\Q{P}:\ct{Q}_P\op\ftr\ISL$$
of a doctrine $P:\ct{C}\op\ftr\ISL$, as introduced in
\cite{MaiettiME:quofcm,MaiettiME:eleqc}.
Intuitively, $\ct{Q}_P$ is obtained from $P:\ct{C}\op\ftr\ISL$ by
freely adding quotients of the equivalence relations specified by $P$,
while $\Q{P}$ extends $P$ to the new sorts of $\ct{Q}_P$ in an
appropriate way. The two doctrines giving rise to \eff and \ass have
the same domain \pass, the full subcategory of \ass (therefore of
\eff) on the partitioned assemblies. Specifically:
\begin{enumerate}
\item The doctrine \Sb{\eff} of the subobjects on \eff is the
elementary quotient completion of the doctrine $\Wsb{\pass}$ of
variations on \pass. The intuition about a doctrine of the form
\Wsb{\ct{A}} for a category \ct{A} dates back to the original paper
\cite{LawvereF:adjif} and the term ``variation'' was introduced in
\cite{GrandisM:weasec}, see \exx\ref{ltae}(b).
\item The doctrine \Ssb{\ass} of strong subobjects on \ass is the
elementary quotient completion of the boolean doctrine \stm on \pass
which is the composition of the powerset functor with the global
section functor $\Gamma:\pass\ftr\Set$.
\end{enumerate}

After showing some general transfer principles describing how
the validity of choice principles and the principles \fct and
\wttct transfers from a doctrine $P$ to its elementary quotient
completion $\Q{P}$ we conclude that:
\begin{enumerate}
\item The doctrine \Ssb{\ass} satisfies only \wttct (but not \fct) as
a direct consequence of the validity of \wttct in \stm. Thanks to an
adjoint situation between \stm and \Wsb{\pass}, we can prove that \wttct 
is inherited by \Wsb{\pass}. And this is strengthened to the full
validity of \fct in \Wsb{\pass} by choice principles.
\item In the logic of \eff the validity of \fct and of choice
principles on partitioned assemblies is a direct consequence of
the validity of corresponding principles on the doctrine \Wsb{\pass}. 
Also the fact that the logic of \eff extends Kleene's realizability
interpretation of Intuitionistic Arithmetic \cite{KleeneS:intint} is
again inherited by \Wsb{\pass}. 
\end{enumerate}
These results on \eff are to be compared with the original
construction of \eff via the tripos-to-topos construction in
\cite{HylandJ:trit} applied to a hyperdoctrine with domain the
category \Set of sets and functions. That hyperdoctrine does not
validate Intuitionistic Arithmetic (neither does it extend Kleene's
realizability interpretation of Intuitionistic Arithmetic!) but
nevertheless it produces the topos \eff whose subobject doctrine
does.

Section~\ref{rcp} collects basic notions about elementary doctrines
$P:\ct{C}\op\ftr\ISL$, introduced in
\cite{LawvereF:adjif,LawvereF:equhcs} as well as the construction of
the elementary quotient completion $\Q{P}:\ct{Q}_P\op\ftr\ISL$. In
section~\ref{scp} we recall some transfer results for some logical
principles between an elementary doctrine and its elementary quotient
completion including a characterization of the doctrine of variations
via a choice principle. In section~\ref{ads} we introduce arithmetic
doctrines, which are doctrines with a parameterized natural number
object which satisfy induction in the sense of the logic determined by
$P$, and we prove that the property of being arithmetic transfers from
suitable doctrines to their elementary quotient completions. We also
prove transfer results for \fct and \wttct. In section~\ref{eqcopa} we
show that the doctrine of subobjects on \eff is the quotient
completion of the doctrine of variations on \pass, and that the
doctrine of strong subobjects on \ass is the quotient completion of
the doctrine \stm on \pass. We then apply the general transfer
principles proved before to deduce the validity of \fct in \Sb{\eff}
and the validity of \wttct in \Ssb{\ass}. Finally in section~\ref{kri}
we justify why \Sb{\eff} extends Kleene's realisability interpretation
of Intuitionistic Arithmetic as a consequence of the facts that
\Wsb{\pass} does so and that \Sb{\eff} is the elementary quotient
completion of $\Wsb{\pass}$.

We would like to thank the referees for their very valuable comments
which were very useful to improve the presentation.

\section{Elementary quotient completions: a brief recap}\label{rcp}

In this section we review some notions and results about elementary
doctrines and their elementary quotient completion, which was
introduced in \cite{MaiettiME:quofcm,MaiettiME:eleqc} and studied
extensively in a series of papers which will be mentioned in due
course.

Let \ct{C} be a category with binary products 
\[\xymatrix{A_1&A_1\times A_2\ar[l]_{\fst}\ar[r]^{\snd}&A_2}\]
for every pair of objects $A_1$ and $A_2$ in \ct{C},
and a terminal object $T$.
Recall from \cite{MaiettiME:quofcm,MaiettiME:eleqc} 
that a \dfn{primary doctrine} on \ct{C} is an indexed
inf-semilattice $P:\ct{C}\op\ftr\ISL$, \ie a (contravariant) functor
$P:\ct{C}\op\ftr\Ct{Pos}$ such that each poset $P(C)$ is an
$\Land$-semilattice and for every arrow $f:A\to B$ in \ct{C} the
monotone map $\fp{f}:P(B)\to P(A)$ is a $\Land$-homomorphism---note
the reversed direction!---and one declares a primary doctrine
\dfn{elementary} when, for every object $A$ in \ct{C}, there is an
object $\delta_A$ in $P(A\times A)$ such that for every arrow $e$ of
the form $\ple{\fst,\snd,\snd}:X\times A\to X\times A\times A$ in
\ct{C}, the assignment 
\[\D_{e}(\alpha)\colon=
\fp{\ple{\fst,\snd}}(\alpha)\Land \fp{\ple{\snd,\pr_3}}(\delta_A)\]
for $\alpha$ in $P(X\times A)$ determines a left adjoint to 
the map $\fp{e}:P(X\times A\times A)\to P(X\times A)$.

Elementary doctrines are the cloven Eq-fibrations of
\cite{JacobsB:catltt} and, as explained in \loccit, there is a
deductive logical calculus associated with them: it is the fragment
of Intuitionistic Logic with conjunctions and equality over a type
theory with a unit type and the binary product type constructor. From
now on, we shall employ the logical language introduced in \loccit and
often write
\[
a_1:A_1,\ldots,a_k:A_k\mid \phi_1(a_1,\ldots,a_k),\ldots,
\phi_n(a_1,\ldots,a_k)\vdash \psi(a_1,\ldots,a_k)
\]
in place of
\[\phi_1\Land\ldots\Land\phi_n\leq \psi\]
in $P(A_1\times\ldots\times A_k)$. Note that, in line with \loccit,
$\delta_A(a,a')$ will be written as $a:A,a':A\mid a\eq{A}a'$. Also we
write $a:A\mid\alpha(a)\vdashv\beta(a)$ to abbreviate the two
facts that $a:A\mid\alpha(a)\vdash\beta(a)$ and
$a:A\mid\beta(a)\vdash\alpha(a)$.

\begin{exas}\label{ltae}
\noindent(a)
The doctrine of subobjects on a category \ct{C} with finite
limits will be denoted as $\Sb{\ct{C}}:\ct{C}\op\ftr\ISL$---the
elementary structure is provided by the diagonal arrows.

\noindent(b)
Another example is provided by the \dfn{doctrine of variations}
$\Wsb{\ct{S}}:\ct{S}\op\ftr\ISL$ of \ct{S}, where \ct{S} is a category
with binary products and weak pullbacks. The fibre on the object $A$
in \ct{S} is the poset reflection of the comma category
$\ct{S}/A$, see \cite{GrandisM:weasec}, the action on arrows is
given by weak pullbacks. 
\end{exas}
The categorical approach makes it possible to express precisely how
the doctrines are related as category theory suggests directly what
``homomorphisms of doctrines'' should be. In fact, one introduces the
2-category \EH of elementary doctrines which has
\begin{description}
\item[1-arrows $(F,b):P\to R$] pairs $(F,b)$ where
$F:\ct{C}\to\ct{D}$ is a functor and $b:P\ntr R\circ F\op$ is a
natural transformation as in the diagram
\[
\xymatrix@C=4em@R=1em{
{\ct{C}\op}\ar[rd]^(.4){P}_(.4){}="P"\ar[dd]_{F\opp}&\\
& {\ISL}\\
{\ct{D}\op}\ar[ru]_(.4){R}^(.4){}="R"&\ar"P";"R"_b^{\kern-.4ex\cdot}}
\]
where the functor $F$ preserves products and, for every object $A$ in
\ct{C}, the functor $b_A:P(A)\to R(F(A))$ preserves finite infima and 
\[
b_{A\times A}(\delta_A)= R_{\ple{F(\fst),F(\snd)}}(\delta_{F(A)});
\]
\item[2-arrows $\theta:(F,b)\to(G,c)$] natural
transformations $\theta:F\ntr G$ such that for every $A$ in \ct{C} and
every $\alpha$ in $P(A)$, one has that
$b_A(\alpha)\leq_{F(A)} R_{\theta_A}(c_A(\alpha))$.
\end{description}

\begin{exas}\label{tae}
Given a category \ct{C} with
products and pullbacks, one can consider the two indexed posets:
that of subobjects $\Sb{\ct{C}}:\ct{C}\op\ftr\ISL$ and that
of variations 
$\Wsb{\ct{C}}:\ct{C}\op\ftr\ISL$. Recall that $\Wsb{\ct{C}}(A)$ is
the poset reflection of the comma category $\ct{C}/A$. Its inclusion
in $\Sb{\ct{C}}(A)$ extends to a 1-arrow from \Sb{\ct{C}} to
\Wsb{\ct{C}}.
\end{exas}

Recall that a category \ct{C} with binary products is
\dfn{weakly cartesian closed} if 
for every pair of objects $A$ and $B$ there is an object $W$ and an
arrow $\ev:W\times A\to B$ such that for every $f:C\times A\to B$
there is $g:C\to W$ with $\ev(g\times \id{A})=f$. Since a category
\ct{C} is cartesian closed when every mediating arrow $g$ in the
condition above is unique, we refer to $W$ as a \dfn{weak exponential}
of $B$ with $A$ and to the arrow $\ev:W\times A\to B$ as a
\dfn{weak evaluation}.

The category \Set is cartesian closed, while the category of
topological spaces and continuous functions is notoriously not
cartesian closed, but it is weakly cartesian closed, see
\cite{RosoliniG:loccce}. 

A \dfn{\whyper} $P:\ct{C}\op\ftr\ISL$ is an elementary doctrine such that
\begin{enumerate}\thmitem
\item \ct{C} is weakly cartesian closed;
\item $P$ factors through the category \Hey of Heyting algebras
and Heyting algebras homomorphisms;
\item for every product projection $\fst:A\times B\to A$ the
monotone map $\fp{\fst}$ has a left adjoint
$\D_{\fst}:P(A)\to P(A\times B)$ and a right adjoint
$\B_{\fst}:P(A)\to P(A\times B)$
\item 
for every arrow $f:X\to A$ 
the canonical inequalities $\D_{\fst'}\fp{f\times \id{B}}\le
\fp{f}\D_{\fst}$ and $\fp{f}\B_{\fst} \le
\B_{\fst'}\fp{f\times\id{B}} $, where $\fst:A\times B\to A$ and
$\pr':X\times B\to X$ are projections, are equalities. 
\end{enumerate}
A \whyper $P:\ct{C}\op\ftr\ISL$ is a \dfn{\hyper} if \ct{C} is
cartesian closed.

When $P:\ct{C}\op\ftr\ISL$ is a \whyper we may write that
$P:\ct{C}\op\ftr\Hey$. Also we shall refer to condition (iv) as the
\dfn{Beck-Chevalley condition}. 

Similarly to the case of elementary doctrines, in line with
\cite{JacobsB:catltt} and \cite{PittsAM:catl}, one can associate a
deductive logical calculus to a \hyper $P:\ct{C}\op\ftr\Hey$:
it is a predicate calculus with equality and a lambda constructor over
a type theory with a unit type, a binary-product type constructor and a
function type constructor. We shall employ the following notation:
Given a term $(a:A,c:C\mid t:B)$ in \ct{C}, the term 
$(c:C\mid\la a:A.t:W)$ is such that the terms $(a:A,c:C\mid t:B)$ and
$(a:A,c:C\mid\ev(\la a:A.t,a):B)$ are equal.
Also a term $(c:C\mid s:W)$ in \ct{C} is equal to the term
$(c:C\mid \la a:A.\ev(s,a):W)$.
Given the well formed formulas $a:A\mid\phi(a)$ and $a:A\mid \psi(a)$
we write \[a:A\mid \phi(a)\lor\psi(a)\qquad a:A\mid
\phi(a)\Implies\psi(a)\] to denote joins and Heyting implication. The
least element will be $a:A\mid \ff$. As is customary, we abbreviate
$a:A\mid \phi(a)\Implies\ff$ with $a:A\mid \neg\phi(a)$. 
For a projection $\fst:A\times B\to A$ and for $\phi$ in
$P(A\times B)$ we shall write $\D_{\fst}(\phi)$ and
$\B_{\fst}(\phi)$ in $P(A)$ as 
\[a:A\mid\Exists b:B.\phi(a,b)\qquad a:A\mid\Forall b:B.\phi(a,b).\]

\begin{rem}
There is instead a radical difference in case $P:\ct{C}\op\ftr\Hey$ is
a \whyper---and in some sense this shows the usefulness of the
categorical presentation. The weakened condition, stripped of
uniqueness, allows to introduce a $\lambda$-notation, but in general
the terms $(c:C\mid s:W)$ and $(c:C\mid \la a:A.\ev(s,a):W)$ do not
coincide, and more importantly, it is not possible to substitute
\emph{inside} a $\lambda$-term. So for a \whyper we shall use all the
above but with no reference to $\lambda$-terms, namely concerning
function types we just use the evaluation constructor.
\end{rem}

\begin{rem}\label{nondipendedaeval}
If $\ev:W\times A\to B$ and $\ev':W'\times A\to B$ are two weak
evaluation maps, then 
\[f:W\vdash \Exists f':W'.\Forall a:A.[\ev(f,a)=\ev'(f',a)]\]
\end{rem}

It is easy to see that, for $P:\ct{C}\op\ftr\Hey$ a
\whyper on \ct{C}, for every arrow $f:A\to B$ the monotone map
$\fp{f}:P(B)\to P(A)$ has 
a left adjoint $\D_f:P(A)\to P(B)$ and a right adjoint $\B_f$ that
send $a:A\mid \alpha(a)$ in $P(A)$ respectively to 
\[b:B\mid \Exists a:A.
\left[\left[f(a)\eq{B}b\right]\Land \alpha(a)\right]
\quad\mbox{ and }\quad
b:B\mid \Forall a:A. 
\left[\left[f(a)\eq{B}b\right]\Implies \alpha(a)\right]\]

We shall employ logical wording to mark certain situations in 
an elementary doctrine $P:\ct{C}\op\ftr\ISL$. For the terminal object
$1$ in \ct{C}, we call an element of $P(1)$ a \dfn{sentence}.
For a sentence $\alpha$ in $P$ such that $\tt\leq\alpha$ we write
$\vdash\alpha$.

\begin{exas}\label{propex}
(a) The doctrine \Sb{\ct{C}} in \exx\ref{ltae}(a) is a (weak)
\hyper if and only if \ct{C} is a (weakly) cartesian closed Heyting
category.

\noindent(b) If \ct{C} is (weakly) locally cartesian closed with
finite (weak) coproducts, the doctrine \Wsb{\ct{C}} in
\exx\ref{ltae}(b) is a (weak) \hyper. Since \whypers of the form
\Wsb{\ct{C}} play a central role in the paper, we find it convenient
to denote the left adjoint along $\Wsb{\ct{C}}(f)$ as $\Sigma_f$ and
the right adjoint as $\Pi_f$.
\end{exas}

Let $P:\ct{C}\op\ftr\Hey$ and $R:\ct{C}\op\ftr\Hey$ be
\whypers. Suppose the natural transformations $r:P\ntr R$ is a
1-arrow of doctrines in \EH; $r$ is a right adjoint if there is a
1-arrow of doctrines $l:R\ntr P$ such that $\Id{R}\le r\circ l$ and
$l\circ r\le \Id{P}$. 
This adjoint pair satisfies the \textit{Frobenius reciprocity} if, for
all $A$ in $\ct{C}$, for all $\alpha$ in $P(A)$ and all $\beta$ in
$R(A)$ it holds that
$l_A(\beta)\Land\alpha=l_A(\beta\Land r_A(\alpha))$. In the following
proposition we use superscript to distinguish operations in $P$ from
the corresponding operations in $R$.

\begin{prop}\label{addexteraperdextera}
If a 1-arrow of doctrine $r:P\ntr R$ is a right adjoint, then for
every $\alpha$ in $P(X\times Y)$ 
\[r_X\B^P_{\fst}(\alpha)=\B^R_{\fst}r_{X\times Y}(\alpha).\]
Moreover for every $\gamma$ and $\beta$ in $P(A)$ it holds
\[r_A(\gamma\Implies^P\beta)=r_A(\alpha)\Implies^R r_A(\beta)\]
if and only if the adjoint pair satisfies the Frobenius reciprocity.
\end{prop}

Note that if the adjoint pair is such that $l\circ r = \Id{P}$, then
it satisfies the Frobenius reciprocity.

\prx\ref{addexteraperdextera} proves that right adjoints
commute with right adjoints. 
Oppositely to right adjoints, left adjoints do not commute with
respect to $r:P\ntr R$. We shall see that in our case of interest $r$
commutes with $\D$ exactly when $P$ satisfies a form of choice that we
call \rc, see \thx\ref{car}. Of course, it might be the case that for
some specific formula the property holds, though $P$ does not satisfy
\rc. This motivates the following definition, which is instrumental
for the proofs of the main theorems in section~\ref{eqcopa}.

\begin{defi}\label{skolemformula}
Suppose $P$ is an elementary existential doctrine, $\alpha$ is in
$P(Y\times B)$ and $\epsilon:Y\to B$ is an arrow in \ct{C}. 
We say that $\epsilon$ is a \dfn{\ska for $B$ in $\alpha$} if
$\D_{\fst}\alpha = \fp{\ple{\id{Y},\epsilon}}(\alpha)$, \ie if
$y:Y\mid\Exists b:B.\alpha(y,b)\vdashv\alpha(y,\epsilon(y))$.
\end{defi}

We use the Greek letter $\epsilon$ to denote a Skolem arrow in
view of the strict connection between Skolem terms and
$\epsilon$-terms of Hilbert's $\epsilon$-calculus. 
Here we observe that, if $B$ has Skolem arrows for all formulas, then
$B$ is endowed with an $\epsilon$-operator as defined in \cite{MPR},
which is a stronger property than the Rule of Choice on $B$ introduced
in \dfx\ref{ciaociaociao} (see also \cite{pep16, pep18}).

Doctrines of subobjects are characterized via the notion of
comprehension. Though very general, we shall recall this notion in the
particular case of an elementary doctrine $P:\ct{C}\op\ftr\ISL$. For a
given object $A$ in \ct{C} and an object $\alpha$ in $P(A)$, a
\dfn{weak comprehension} of $\alpha$ is an arrow
$\cmp{\alpha}:X\to A$ in \ct{C} such that 
\[x:X\mid \tt\vdash\alpha(\cmp{\alpha}(x))\]
and, for every $f:Z\to A$ such that 
$z:Z\mid \tt\vdash\alpha(f(z))$,
there is an arrow 
$f':Z\to X$ such that 
$f=\cmp{\alpha}f'$. The arrow $\cmp{\alpha}$ is the \dfn{strong
comprehension} or simply \dfn{comprehension} of $\alpha$ if
\cmp{\alpha} is monic, making the required $f'$ the unique such.

Intuitively, the comprehension arrow represents the inclusion of the
object obtained by comprehending the predicate $\alpha$ over $A$ into
$A$ itself as a form of subtype.

We simply say that the doctrine $P:\ct{C}\op\ftr\ISL$ 
\dfn{has $($weak$)$ comprehensions} 
when every $\alpha$ has a (weak) comprehension arrow. And $P$
\dfn{has full $($weak$)$ comprehensions} if $\alpha\leq\beta$ in 
$P(A)$ whenever $\cmp\alpha$ factors through $\cmp\beta$.

\begin{exa}\label{escomp}
Doctrines of the form \Wsb{\ct{C}} have full weak comprehensions: if
$[f]$ is in $\Wsb{\ct{C}}(A)$ then the
representative $f$ is a weak full comprehension of $[f]$: it is strong
if and only if $f$ is monic. So doctrines of the form \Sb{\ct{C}}
have full comprehensions.
\end{exa}

An elementary doctrine $P:\ct{C}\op\ftr\ISL$ 
\dfn{has comprehensive diagonals} if for every $A$ in \ct{C} the 
diagonal $\Delta_A:A\to A\times A$ is the full comprehension of
$\delta_A$.
It is straightforward to verify that an elementary doctrine has
comprehensive diagonals if and only if any two parallel arrows of
\ct{C}, say $f,g:X\to Y$, are equal whenever $x:X\mid \tt\vdash
f(x)\eq{Y}g(x)$. 

The intuition underlying the construction of the elementary quotient
completion is to add quotients to the domain of the elementary
doctrine with respect to equivalence relations in the fibres of the
doctrine.

In an elementary doctrine $P:\ct{C}\op\ftr\ISL$, if $A$ is an object
in \ct{C}, an object $\rho$ in $P(A\times A)$ is a
\dfn{$P$-equivalence relation on $A$} if it satisfies
\begin{description}
\item[\dfn{reflexivity}] $a:A,a':A\mid a\eq{A}a'\vdash\rho(a,a')$;
\item[\dfn{symmetry}]
$a:A,a':A\mid \rho(a,a')\vdash\rho(a',a)$;
\item[\dfn{transitivity}]
$a:A,a':A,a'':A\mid \rho(a,a')\Land\rho(a',a'')\vdash\rho(a,a'')$.
\end{description}

\begin{exas}\label{exer}
\noindent(a)
For a category \ct{D} with products and pullbacks, consider the
elementary doctrine $\Sb{\ct{D}}:\ct{D}\op\ftr\ISL$ of the subobjects
of \ct{D}. A \Sb{\ct{D}}-equivalence relation is an equivalence
relation in \ct{D}. In particular, \Sb{\Set}-equivalence relations
coincide with the usual notion of equivalence relations. 

\noindent(b)
For a category \ct{C} with products and weak pullbacks, 
consider the elementary doctrine $\Wsb{\ct{C}}:\ct{C}\op\ftr\ISL$ of
the variations. A \Wsb{\ct{C}}-equivalence relation is a
pseudo-equivalence relation in \ct{C}, see \cite{CarboniA:freecl}.
\end{exas}

Given a $P$-equivalence relation $\rho$ on $A$, a
\dfn{$P$-quotient of $\rho$}, or simply a \dfn{quotient} when the
doctrine is clear from the context, is an arrow $q:A\to A/\rho$ in
\ct{C} such that 
\[a;A,a':A\mid\rho(a,a')\vdash q(a)\eq{A/\rho}q(a')\]
and, for every arrow $g:A\to Z$ such that
\[a;A,a':A\mid\rho(a,a')\vdash g(a)\eq{Z}g(a'),\]
there is a unique arrow
$\qu(g):A/\rho\to Z$ 
such that $g=\qu(g) q$. 

We say that such a $P$-quotient is \dfn{stable} if, for every pullback
\[\xymatrix{B\ar[d]_{f'}\ar[r]^{q'}&C\ar[d]^{f}\\A\ar[r]_q&A/\rho}\]
in \ct{C}, the arrow $q'$ is a $P$-quotient. 

For an equivalence relation $\rho$ on $A$, the poset \des{\rho} of
\dfn{descent data} is the sub-poset of $P(A)$ on those $\alpha$ such
that
\[a:A,a':A\mid\alpha(a)\Land\rho(a,a')\vdash \alpha(a').\]

Like for comprehension, it is possible to complete an elementary
doctrine $P:\ct{C}\op\ftr\ISL$ to one with stable quotients of
equivalence relations: the \dfn{elementary quotient completion}
$\Q{P}:\ct{Q}_P\op\ftr\ISL$ of $P$ which was introduced and studied in
\cite{MaiettiME:eleqc,MaiettiME:quofcm,MaiettiME:exacf,Maietti-Rosolini16}.
It is defined as follows
\begin{description}
\item[Objects of $\ct{Q}_P$] $(A,\rho)$ such that $\rho$
is a $P$-equivalence relation on $A$.
\item[Arrows of $\ct{Q}_P$] an arrow $\ec{f}:(A,\rho)\to(B,\sigma)$ is an
equivalence class of arrows $f:A\to B$ in \ct{C} such that 
\[a:A,a':A\mid \rho(a,a')\vdash \sigma(f(a),f(a'))\] in $P(A\times A)$ with
respect to the relation $f\sim g$ which holds if and only if 
\[a:A,a':A\mid \rho(a,a')\vdash\sigma(f(a),g(a')).\]
\item[Composition of $\ct{Q}_P$] that of \ct{C} on representatives.
\item[Identities of $\ct{Q}_P$] are represented by identities of \ct{C}.
\item[The functor $\Q{P}:\ct{Q}_P\op\ftr\ISL$] is defined as
\[\Q{P}(A,\rho)\colon=\des{\rho}.\]
\end{description}

We refer the reader to \cite{MaiettiME:eleqc} for all the
details. We just note that the exact completion in
\cite{CarboniA:somfcr} has a description in terms of the elementary
quotient completion of doctrines: given a category \ct{C} with finite
products and weak pullbacks, the doctrine \Sb{\ct{C}\exl} is
equivalent to the doctrine \Q{\Wsb{\ct{C}}}.

Here we limit ourselves to recall a few properties of the
constructions:
\begin{itemize}
\item the elementary quotients completion has effective quotients: for
an equivalence relation $\sigma$ on $(A,\rho)$, the quotient is given
by \[[\id{A}]:(A,\rho)\to(A,\sigma);\]
\item the equality predicate over $(A,\rho)$ is $\rho$ itself,
\ie $\delta_{(A,\rho)}=\rho$;
\item in case $P$ is a \whyper, the evaluation in $\ct{Q}_P$ 
$[\ev]:(B,\delta_B)^{(A,\delta_A)}\times(A,\delta_A)\to(B,\delta_B)$ 
can be chosen as a weak evaluation $\ev:W\times A\to B$ in \ct{C}, and
$(B,\delta_B)^{(A,\delta_A)}$ is $(W,\theta)$ where $\theta$ in $P(W\times
W)$ is the formula
$t:W,t':W\mid \Forall a:A.\Forall a':A. a\eq{A}a'\Implies
\ev(t,a)\eq{B}\ev(t',a').$
\end{itemize}
It is quite apparent that the elementary structure plays no role in
the definitions of \Q{P}. We refer the reader to
\cite{PasqualiF:cofced, TTT} for an analysis of that.

\section{Some choice principles}\label{scp}

In this section we analyse various forms of choice principle in the
context of existential elementary doctrines.

Let $P:\ct{C}\op\ftr\Hey$ be a \whyper. An element $R$ of
$P(A\times B)$ is often called a \dfn{relation}. We say that a
relation $R$ is \dfn{entire} if
\[a:A\mid \tt\vdash \Exists b:B. R(a,b)\] 
and that it is \dfn{functional} if
\[a:A,b:B,b':B\mid R(a,b)\Land R(a,b')\vdash b\eq{B}b'.\] 
For every arrow $f:A\to B$ the formula in $P(A\times B)$ determined
by \[a:A,b:B\mid f(a)\eq{B}b\] is an entire functional relation,
called the \dfn{$P$-graph} of $f$.

\begin{defi}\label{ciaociaociao}
Let $P:\ct{C}\op\ftr\Hey$ be a \whyper.
\begin{description}
\item[The \dfn{Rule of Unique Choice} \ruc holds in $P$] if for
every entire functional relation $R$ in $P(A\times B)$ there is an
arrow $f:A\to B$ whose $P$-graph is $R$.
\item[The \dfn{Rule of Choice} \rc holds in $P$] if for every
entire relation $R$ in $P(A\times B)$ there is an arrow $f:A\to B$
such that
\[a:A\mid \tt\vdash R(a,f(a)).\]
\item[The \dfn{Rule of Choice} holds on $A$ in $P$] if for every
entire relation $R$ in $P(A\times A)$ there is an arrow $f:A\to A$
such that 
\[a:A\mid \tt\vdash R(a,f(a)).\]
\end{description}
\end{defi}

There are axioms that correspond to \ruc and to \rc respectively.

\begin{defi}\label{AUC}
Let $P$ be a \whyper. Let $A$ be an object of \ct{C}.
We say that \dfn{the Axiom of Unique Choice \auc holds on} $A$
if, for every object $B$ in \ct{C}, for every relation $R$ in
$P(A\times B)$ it is 
\[\Forall a:A.\Existu b:B. R(a,b)\vdash
\Exists f:W.\Forall a:A. R(a,\ev(f,a))\]
where $\ev:W\times A\to B$ is a weak evaluation map.
We say that \dfn{the Axiom of Choice \ac holds on} $A$
if, for every object $B$ in \ct{C}, for every relation $R$ in
$P(A\times B)$ it is 
\[\Forall a:A.\Exists b:B. R(a,b)\vdash
\Exists f:W.\Forall a:A. R(a,\ev(f,a))\]
where 
$\ev:W\times A\to B$ is a weak evaluation map.
When the Axiom of (Unique) Choice holds on every object $A$ in \ct{C},
we say that the \dfn{Axiom of $($Unique$)$ Choice holds in} $P$.
\end{defi}

Clearly, if \ac holds on $A$, then \auc holds on $A$.

Those choice principles are useful to characterize variational
doctrines as shown in \cite{MPR}. That characterization employs also
an adjunction between variational doctrines and an elementary
existential doctrine $P$ with full weak comprehensions as stated in
the following proposition from \loccit.

\begin{prop}\label{aggiunzione}
Suppose $P$ is an elementary existential doctrine on \ct{C} with full
weak comprehensions and comprehensive diagonals. There are arrows of
doctrines
\[\xymatrix@C=5em@R=1em{
\mathcal{C}^{op}\ar@/^/[rrd]^(.35){P}_(.35){}="P"
\ar@/_/@<-1ex>[dd]_{\id{\mathcal{C}}^{op}}="F"&&\\
&& {\ISL}\\
\mathcal{C}^{op}\ar@/_/[rru]_(.35){\Wsb{\mathcal{C}}}^(.35){}="R"&&
\ar@/^/"R";"P"^-{\LL{}\kern.5ex\cdot\kern-.5ex}="b"
\ar@<1ex>@/^/"P";"R"^{\kern-.5ex\cdot\kern.5ex \RR}="c"
\ar@{}"b";"c"|(.55){}}\]
such that $\LL{}\circ\RR=\id{P}$ and
$\id{\Wsb{\ct{C}}}\le\RR\circ\LL{}$.
\end{prop}

For clarity, we recall the construction of the
two natural transformations: For an object $A$ of \ct{C},
$[f:X\to A]$ in \Wsb{\ct{C}}, and $\alpha$ in $P(A)$ it is
\[\LL{[f]}\colon=\D_f(\tt_X)
\qquad\mbox{ and }\qquad
\RR_A(\alpha)=[\cmp{\alpha}].\] 

Note that the conditions $\id{\Wsb{\ct{C}}}\le\RR\circ\LL{}$ and
$\LL{}\circ\RR= \id{P}$ establish that $\RR$ and $\LL{}$ form an
adjoint pair satisfying Frobenius reciprocity. This will be useful to
prove commutativity of $\RR$ and $\Sigma$ for some formulas of $P$.

\begin{rem}
Note that if a \whyper on \ct{C} has comprehensive diagonals and
full weak comprehensions then \ct{C} has weak pullbacks whereas, if
comprehensions are strong, then \ct{C} has pullbacks, see
\cite{MPR}. For this reason we did not assume pullbacks or weak
pullbacks in the formulation of \prx\ref{aggiunzione}.
\end{rem}

\begin{rem}\label{pippopippo}Suppose $P$ is a \whyper on \ct{C}
with full comprehensions and comprehensive diagonals. There is
adjunction situation analogous to the one described in
\prx\ref{aggiunzione} between $P$ and $\Sb{\ct{C}}$,
\ie $\LL{}:\Sb{\ct{C}}\stackrel.\to P$ and
$\RR:P\stackrel.\to\Sb{\ct{C}}$ are such that
%
$\LL{}\circ\RR=\id{P}$ and
$\id{\Sb{\ct{C}}}\le\RR\circ\LL{}$.
\end{rem}
\prx\ref{aggiunzione} and \rmx\ref{pippopippo} together with
\prx\ref{addexteraperdextera} prove the following corollary.

\begin{cor}\label{ddd}
Suppose $P$ is a \whyper on \ct{C} with full weak comprehensions and
comprehensive diagonals, then for every $\alpha$ in $P(X\times Y)$ and
every $\gamma$ and $\beta$ in $P(A)$ it is 
\[\cmp{\B^P_{\fst}(\alpha)}=\B^{\Wsb{\ct{C}}}_{\fst}\cmp{\alpha}\qquad
\cmp{\gamma\Implies^P\beta}=\cmp{\alpha}\Implies^{\Wsb{\ct{C}}}
\cmp{\beta}.\]
Moreover, if comprehensions are strong, it also holds that
\[\cmp{\B^P_{\fst}(\alpha)}=\B^{\Sb{\ct{C}}}_{\fst}\cmp{\alpha}\qquad
\cmp{\gamma\Implies^P\beta}=\cmp{\alpha}\Implies^{\Sb{\ct{C}}}
\cmp{\beta}\]
where superscripts distinguish operations between $P$ and
$\Wsb{\ct{C}}$ and between $P$ and $\Sb{\ct{C}}$.
\end{cor}

Among all doctrines, the subobject doctrines of the form \Sb{\ct{C}}
are characterized by the fact that they satisfy \ruc (see
\cite{JacobsB:catltt}), while variational doctrines of the form
\Wsb{\ct{C}} are characterized by the fact that they satisfy \rc (see
\cite{MPR}). Since we shall refer to this characterization repeatedly
in the special case of elementary existential doctrines, we state it
explicitly in the next theorem. We refer the reader to \cite{MPR} for
a proof.

\begin{thm}\label{car}
Suppose $P:\ct{C}\op\ftr\ISL$ is a \whyper.
\begin{enumerate}\thmitem
\item The doctrine $P$ is equivalent to \Sb{\ct{C}} if and only if $P$
has comprehensive diagonals, full weak comprehensions and 
\ruc holds in $P$.
\item The doctrine $P$ is equivalent to \Wsb{\ct{C}} if and only if
$P$ has comprehensive diagonals, full weak comprehensions and
\rc holds in $P$.
\item If the doctrine $P$ has comprehensive diagonals, full weak
comprehensions, then it is equivalent to \Wsb{\ct{C}} if and only if
the inequality $\id{\Wsb{\ct{C}}}\le\RR\circ\LL{}$ is in fact an
equality.
\end{enumerate}
\end{thm}

\begin{proof}
See \cite{JacobsB:catltt} for the proof of (i); see \cite{MPR} for
those of (ii) and (iii).
\end{proof}

\prx4.11 in \cite{Maietti-Rosolini16} states that in any \whyper $P$
with comprehension \rc holds if and only if \ruc holds in $\Q{P}$. So
\thx\ref{car} immediately gives the following result.

\begin{cor}\label{car3}
Let $P:\ct{C}\op\ftr\ISL$ be a \whyper with full weak comprehensions
and comprehensive diagonals. The doctrine $P$ is equivalent to
\Wsb{\ct{C}} if and only if the doctrine \Q{P} is equivalent to
\Sb{\ct{Q}_p}.
\end{cor}

Observe that in a \whyper with full weak comprehensions
the validity of \ruc implies that of \auc, as well
as the validity of \rc implies that of \ac. This can be proved
by translating in the internal language of \whypers 
the proofs in \cite{tamc}. Therefore, if \Sb{\ct{C}} is a \whyper
then \auc holds in \Sb{\ct{C}}, and if \Wsb{\ct{C}}
is a \whyper, then \ac holds in \Wsb{\ct{C}}. Moreover the proof in
\prx6.5 in \cite{Maietti-Rosolini16} proves also the following.

\begin{prop}\label{car2}
Let $P:\ct{C}\op\ftr\ISL$ be a \whyper.
\begin{enumerate}\thmitem
\item If \ac holds in $P$, then \auc holds in the quotient completion
\Q{P}.
\item If \auc holds \Q{P} and $P$ has full weak comprehensions, then
\ac holds in $P$.
\end{enumerate}
\end{prop}




\section{\Aritm doctrines}\label{ads}

The aim of this section is to show that the elementary quotient
completion inherits the validity of Formal Church's Thesis from the
doctrine on which it is performed. To this purpose we first briefly
show some preliminary results concerning primary doctrines 
equipped with a natural numbers object.

Recall \cite{LambekJ:inthoc} that, in a category
\ct{C} with binary products, a
\dfn{parameterized natural number object} (\nno) is an object \NN
together with two arrows $\oo:1\to\NN$ and $\ss:\NN\to\NN$ such that
for every $A$ and $X$ and every pair of arrows $a:A\to X$ and 
$f:X\to X$ there is a unique arrow $k:A\times \NN\to X$ such that the
following diagram
\begin{equation}\label{drc}
\vcenter{\xymatrix{A\ar[rrd]_-{a}\ar[rr]^-{\ple{\id{A},\oo}}&&
A\times \NN\ar[rr]^-{\id{A}\times \ss}\ar[d]^-{k}&&
A\times \NN\ar[d]^-{k}\\
&&X\ar[rr]_-{f}&&X}}
\end{equation}
commutes.

Let $P:\ct{C}\op\larr\ISL$ be a primary doctrine, and suppose that 
$(\NN,\oo,\ss)$ is a \nno in
\ct{C}. We say that the \nno satisfies \dfn{induction in} $P$ when
for every $A$ in \ct{C} and $\phi$ in $P(A\times\NN)$, if
$a:A\vdash\phi(\oo)$ and $a:A,m:\NN\mid\phi(m)\vdash\phi(\ss(m))$,
then also \[a:A,n:\NN\vdash\phi(n).\]

\begin{rem}
There is a weakened version of the notion of \nno when, for pairs
$(a,f)$, the mediating arrow $k$ is not necessarily unique with the
commutation property. There is no point to consider the weak version
here because, if $P:\ct{C}\op\larr\ISL$ is a \whyper with
comprehensive diagonals, and $(\NN,\oo,\ss)$ is a \wnno which
satisfies induction in $P$, then it is a \nno in \ct{C}. To see this,
given arrows $a:1\to A$ and $f:A\to A$, suppose that $k:A\times\NN\to
A$ and $h:A\times\NN\to A$ make the diagram (\ref{drc}) commute. So
\[\vdash k(\oo)\eq{\NN}h(\oo)\quad\mbox{ and }\quad
a:A,n:N\mid k(a,n)\eq{A}h(a,n)\vdash k(a,\ss(n))\eq{A}h(a,\ss(n)).\] 
By induction $a:A,n:\NN\vdash k(a,n)\eq{A}h(a,n)$, and $k=h$ since
diagonals are comprehensive.
\end{rem}

We are interested in studying the behavior of \aritm doctrines with
respect to the notion of elementary quotient completion. Since all our
examples and applications concern elementary doctrines with
comprehensive diagonals, from now on we will consider only this class
of doctrines, and \aritm doctrines within.

\begin{rem}
Induction takes a more familiar form when the doctrine $P$
bears sufficient structure to express it. In case $P$ is a \whyper,
the \nno $(\NN,\oo,\ss)$ satisfies
induction if and only if, for every $A$ in \ct{C}
and $\phi$ in $P(A\times\NN)$,
\[\vdash\Forall a:A.\left[\left[
\phi(\oo)\Land\Forall m:\NN.[\phi(m)\Implies\phi(\ss(m))]\right]
\Implies\Forall n:\NN.\phi(n)\right]\]
\end{rem}

A \whyper $P:\ct{C}\op\larr\ISL$ with a \nno which satisfies
induction is said \dfn{\aritm}.

\begin{prop}\label{comp-to-aritm} 
Suppose $\ct{C}$ has a \nno. If $P:\ct{C}\op\larr\ISL$ is a \whyper
with full weak comprehensions and comprehensive diagonals, then $P$
is \aritm. 
\end{prop}

\begin{proof}
Suppose that $\phi$ in $P(A\times \NN)$ is such that
$a:A\vdash\phi(\oo)$ and
$a:A,m:\NN\mid\phi(m)\vdash\phi(\ss(m))$. Consider a weak
comprehension $\cmp{\phi}:X\to A\times\NN$ of $\phi$. By the property
of weak comprehension, the condition $a:A\vdash\phi(\oo)$ implies that
$\id{A}$ factors through the weak pullback of $\cmp{\phi}$ along
$\ple{\id{A},\oo}:A\to\NN$, while the condition
$a:A,m:\NN\mid\phi(m)\vdash\phi(\ss(m))$ implies that $\cmp{\phi}$
factors through the weak pullback of $\cmp{\phi}$ along
$\id{A}\times\ss:A\times\NN\to A\times\NN$. The resulting commutative
diagram is 
\[\xymatrix{
1\ar[rd]_-{\ple{\id{A},\oo}}\ar[r]&
X\ar[rr]\ar[d]^-{\cmp{\phi}}&&X\ar[d]^-{\cmp{\phi}}\\
&A\times\NN\ar[rr]_-{\id{A}\times \ss}&&A\times\NN
}\]
The universal property of \NN, gives a section of
$\cmp{\phi}$. Fullness of comprehensions completes the proof. 
\end{proof}

\begin{exa}\label{exsko}
Suppose that \ct{C} has a \nno. If \Sb{\ct{C}} is a \whyper, then
it is also \aritm. If \Wsb{\ct{C}} is a \whyper, then is also \aritm. 
\end{exa}

\begin{lem}\label{eqc-aritm0} If $P$ is an elementary doctrine
with comprehensive diagonals, then \ct{C} has a \nno if and only if
$\ct{Q}_P$ has a \nno. 
\end{lem}

\begin{proof} 
We shall employ \lmx5.7 in \cite{MaiettiME:quofcm} and see \ct{C} as
the full subcategory of $\ct{Q}_P$ on the objects of the form
$(A,\delta_A)$. Suppose $((\NN,\rho),[\oo],[\ss])$
is a \nno in $\ct{Q}_P$. The arrow $[\id{\NN}]:(\NN,\delta_\NN)\to (\NN,\rho)$
makes $((\NN,\delta_\NN), [\oo], [\ss])$ a \nno in $\ct{Q}_P$.
Conversely, suppose
$(\NN,\oo,\ss)$ is a \nno in
\ct{C}. Then it is easy to check that
$((\NN,\delta_\NN), [\oo],[\ss])$
is a \nno in $\ct{Q}_P$.
\end{proof}

\begin{cor}\label{eqc-aritm} Suppose $P$ is an elementary doctrine
with comprehensive diagonals. The doctrine $P$ is \aritm if and only
if the doctrine \Q{P} is \aritm. 
\end{cor}

\begin{proof} Immediate from \lmx5.7 in \cite{MaiettiME:quofcm} and
\lmx\ref{eqc-aritm0}.
\end{proof}

Let $P$ be an \aritm \whyper and let $W$ be a weak exponential of
$\NN$ over $\NN$ with weak evaluation $\ev:W\times\NN\to\NN$.
One can develop standard recursion theory as the
operations of sum and product of pair of natural numbers can be
introduced using weak exponentials and the \nno structure. So one can
introduce the standard Kleene primitive recursive arrows for test and output
$\TU:\NN\times\NN\times\NN\to\NN$ and $\UU:\NN\to\NN$.

For the rest of the section, $P$ is assumed to be an \aritm \whyper
on $\ct{C}$. So in particular $\ct{C}$ is weakly cartesian closed.

\begin{nota}
For $R$ is in $P(\NN\times\NN\times A)$, write 
$\kleene{R}(e,x,y,a)$ in $P(\NN\times\NN\times\NN\times A)$ for the
formula
\[\TU(e,x,y)\eq{\NN}\ss(\oo)\Land R(x,\UU(y),a).\]
Write \kleene{\ev} in $P(\NN\times\NN\times\NN\times W)$ for
\kleene{S} where $S(x,n,f)$ is
\[n\eq{\NN}\ev(f,x).\] 
And write $\ttch_\ev(f)$ in $P(W)$ for the formula
\[\Exists e:\NN.\Forall x:\NN.\Exists y:\NN.
\kleene{\ev}(e,x,y,f)\]
\end{nota}

\begin{lem}\label{CTnotdepends} In every \aritm doctrine, if $W$ and
$W'$ are weak exponentials of \NN with \NN, with
corresponding weak evaluations $\ev:W\times\NN\to\NN$ and
$\ev':W'\times\NN\to\NN$, then
\[\vdash\Forall f:W.\ttch_w(f)\Iff\Forall g:W'.\ttch_{w'}(g).\] 
\end{lem}

\begin{proof}Immediate consequence of
\rmx\ref{nondipendedaeval}.\end{proof}

By \lmx\ref{CTnotdepends}, we can discard the index in $\ttch_\ev$ and
simply write $\ttch$.

\begin{defi}\label{churchThesis}
Let $P:\ct{C}\op\larr\ISL$ be an \aritm doctrine with comprehensive
diagonals. We say that
\begin{enumerate}
\item the \dfn{$($formal$)$ Type-theoretic Church's Thesis}
holds in $P$ if for some weak evaluation
$ev:\NN\times W\to\NN$ it is
\[\vdash \Forall f:W. \ttch_\ev(f);\]
\item the \dfn{$($formal$)$ Church's Thesis} holds in $P$
if, for every $R$ in $P(\NN\times\NN)$,
\[\vdash\Forall x:\NN.\Exists n:\NN.R(x,n)\Implies
\Exists e:\NN.\Forall x:\NN.\Exists y:\NN.\kleene{R}(e,x,y)\]
\end{enumerate}
\end{defi}

\begin{rem}
Note that, by \lmx\ref{CTnotdepends}, any evaluation can be chosen in
the formula $\Forall f:W. \ttch_\ev(f)$. Because of that, we shall
refer to such a sentence as $(\wttct)$. On the other hand, Church's
Thesis is a schema of formulas $\fct_R$ as $R$ varies in
$P(\NN\times\NN)$; and we may abbreviate the statement that Church's
Thesis holds in $P$ by writing that $(\fct)$ holds in $P$.
\end{rem}

\begin{prop}\label{CT}
Suppose $P:\ct{C}\op\larr\ISL$ is an \aritm doctrine with
comprehensive diagonals. The following hold:
\begin{enumerate}\thmitem
\item
The schema $(\fct)$ holds in $P$ if and only if the schema $(\fct)$
holds in \Q{P}.
\item
The sentence $(\wttct)$ holds in $P$ if and only if the sentence
$(\ttct)$ holds in \Q{P}.
\end{enumerate}
\end{prop}

\begin{proof}
By \lmx\ref{eqc-aritm}, $(\NN,\oo,\ss)$ is a \nno in \ct{C} if and
only if $((\NN,\delta_\NN), [\oo],[\ss])$ is a \nno in $\ct{Q}_P$.

\noindent(i) The claim is proved since, for a
fixed $R$ in $P(\NN\times\NN)$ the formula $\fct_R$ in it is built
using only quantifications, finite conjunctions and the
equality predicate over $\NN$, and these operations of $\Q{P}$ over
any finite power of $(\NN,\delta_\NN)$ are the restriction of those of
$P$ over the corresponding finite power of $\NN$.\\ 
\noindent(ii) The ($\Leftarrow$) direction follows from \lmx5.7 in
\cite{MaiettiME:quofcm}: 
$P$ is equivalent to the restriction of \Q{P} to the
subcategory of $\ct{Q}_P$ on objects of the form $(A,\delta_A)$.
For the other direction, note that by \prx6.7 in
\cite{MaiettiME:quofcm} an arrow $w:\NN\times W\to\NN$ is a weak
evaluation in \ct{C} if and only 
if $[w]:(\NN,\delta_\NN)\times(W,\theta)\to (\NN,\delta_\NN)$ is an
evaluation map in $\ct{Q}_P$ where $\theta$ is an appropriate
$P$-equivalence relation over $W$. The claim is proved since the
formula $\wttct$ is built using only the universal quantification,
finite conjunctions and the equality predicate over $(\NN,\delta_\NN)$
and $(W,\theta)$, and these operations of $\Q{P}$ are the restriction
of those of $P$ as $\Q{P}(W,\theta)\subseteq P(W)$. 
\end{proof}

\begin{cor}\label{completions-CT}
Suppose \ct{C} is such that \Wsb{\ct{C}} is \aritm. 
\begin{enumerate}\thmitem
\item
The schema $(\fct)$ holds in \Wsb{\ct{C}} if and only if the schema
$(\fct)$ holds in \Sb{\ct{C}\exl}.
\item
The sentence $(\wttct$) holds in \Wsb{\ct{C}} if and only if the
schema $(\ttct)$ holds in \Sb{\ct{C}\exl}.
\end{enumerate}
\end{cor}

\begin{proof}
It is a direct consequence of \crx\ref{car3} and \prx\ref{CT}.
\end{proof}

It is well known that if the validity of $(\ttct)$ in a theory
implies the validity of $(\fct)$ in the presence of choice principles
for total relations on natural numbers. 

Some forms of choice are transferred via the elementary quotient
completion under suitable assumptions on the doctrine.

Let $P:\ct{C}\op\larr\Hey$ be a \hyper with comprehensive
diagonals. \prx\ref{car2} says that the elementary
quotient completion necessarily transfers \ac to \auc. Thus \ac is
in general not preserved by the completions discussed so
far. Nevertheless there are some instances of \ac
restricted to specific objects of the domain of the doctrine as in
\dfx\ref{AUC}.

\begin{prop}\label{RCA} 
Let $P:\ct{C}\op\larr\Hey$ be a \whyper with
comprehensive diagonals and let $A$ be an object of \ct{C}. The
doctrine $P$ satisfies \ac on $A$ if and only if the doctrine \Q{P}
satisfies \ac on $(A,\delta_A)$
\end{prop}

\begin{proof}
It follows from \lmx5.7 in \cite{MaiettiME:quofcm} and from the fact
that, if $w:A\times W\to A$ is a weak evaluation in \ct{C}, then
$[w]:(A,\delta_A)\times(W,\rho)\to(A,\delta_A)$ is an evaluation map
in $\ct{Q}_P$, where $\rho$ is a suitable $P$-equivalence
relation. Moreover quantifiers of \Q{P} are those of $P$ and
$\des{\rho}\subseteq P(W)$.
\end{proof}

When $P:\ct{C}\op\larr\ISL$ is \aritm, we say that $P$ satisfies the
\dfn{Countable Axiom of Choice} \rcn when $P$ satisfies the \ac on the
\nno of \ct{C}.

As an immediate corollary of \prx\ref{RCA} and
\lmx\ref{eqc-aritm} we have the following. 
\begin{cor}\label{RCN}
Let $P:\ct{C}\op\larr\ISL$ be an \aritm doctrine. The doctrine $P$
satisfies \rcn if and only if \Q{P} satisfies \rcn.
\end{cor}

\begin{prop}\label{choiceCT}
Suppose $P$ is an \aritm doctrine on \ct{C}. If $P$ satisfies both
\rcn and $(\wttct)$, then $P$ satisfies $(\fct)$. 
\end{prop}

\section{Elementary quotient completions on partitioned
 assemblies}\label{eqcopa}

We are finally in a position to analyze the realizability model
offered by the effective topos and various doctrines related to
it. For a detailed presentation of the categorical structure of
realizability we refer the reader to \cite{OostenJ:reaait}; here we
restrict ourselves to give just the essential details needed for our
purposes.

Although most of the development could be performed relative to an
arbitrary partial combinatory algebra, we shall refer only to the
partial combinatory algebra which is Kleene's first model \Kl1 on
the natural numbers, with the usual notation \tur{e} for the $e$-th
partial recursive function. We shall write \epl{n,m} for a fixed
recursive encoding of pairs and \fs{k} and \sn{k} for the (unique)
pair of numbers such that $k=\epl{\fs{k},\sn{k}}$.

Recall the category \ass of assemblies and its full
subcategory of partitioned assemblies from \cite{CarboniA:catarp}. An
\dfn{assembly} is a pair $(P,T)$ where $P$ is a set and $T\subseteq
P\times\N$ is a total relation from $P$ to \N, \ie for every element
$x\in P$ there is a number $n\in\N$ such that
$x\mathrel{T}n$.\footnote{The name assembly refers to the way the
relation $T$ ``assembles'' the elements of $P$ within subsets
$T_n\colon=\{x\in P\mid x\mathrel{R}n\}$, possibly overlapping.}
An arrow $f:(P,T)\to(P',T')$ of assemblies is a function $f:P\to P'$
such that for some $t\in\N$
$t$ \dfn{tracks} $f$, \ie
for every $x\in P$ and every $n\in\N$,
if $x\mathrel{T}n$, then $f(x)\mathrel{T'}\tur{t}(n)$. Arrows compose
as functions. The category \ass is a quasitopos, see
\cite{eff}. In particular, a strong subobject of $(P,T)$ in \ass is
represented by an inclusion
$\xymatrix@1@=2ex{\id{P}\rst_{X}:(X,T\cap(X\times\N))\ \ar@{>->}[r]&(P,T)}$
for some (unique) subset $X$ of $P$.
Also, since the terminal assembly $\spe1=(\{0\},\{(0,0)\})$ is a
generator in \ass, the global-section functor
$\Gamma=\hom_{\pass}(\spe1,\blank):\pass\ftr\Set$
is (isomorphic to) the forgetful functor that sends $(P,T)$ to $P$ and
$f$ to itself.

An assembly $(P,T)$ is \dfn{partitioned} if $T$ is
single-valued (hence $T$ is a function from $P$ to \N).\footnote{The
past participle partitioned refers to the fact that the assembled
subsets $T_n$ of $P$ are disjoint, hence form a partition of $P$.}
The full subcategory of \ass on partitioned assemblies is written
\pass.

The category \pass of partitioned assemblies has finite limits,
finite coproducts, weak exponentials, and a \nno, see
\cite{CarboniA:somfcr,OostenJ:reaait}.

The terminal assembly \spe1 is partitioned. 
The product of the two partitioned assemblies $(P,T)$ and $(M,S)$
can be chosen as $(P\times M,T\otimes S)$ where 
$(T\otimes S)(x,y)\colon=\epl{T(x),S(y)}$.
A weak exponential of $(P,T)$ with $(M,S)$ is $(W,V)$ where
\[W\colon=\{(f,t)\in P^M\times\N\mid t\mbox{ tracks }f\}\]
and $V(f,t)\colon=t$; the weak evaluation
$\ev:(W,V)\times(P,T)\to(M,S)$ is given by the function 
$\ev:W\times P\to M$ defined as $\ev((f,t),x)\colon=f(x)$ which is
tracked by a code for the recursive function
$k\mapsto\tur{\fs{k}}(\sn{k})$.
The \nno is determined on the partitioned assembly $(\N,\id{\N})$.

\begin{rem}\label{rcar}
Recall that $\ass\equiv\pass\rgl$ and that the exact
completion $\pass\exl$ is the effective topos \eff, see
\cite{RobinsonE:colcet,CarboniA:somfcr,OostenJ:reaait}. A crucial
point to see this is that every partitioned assembly is projective
with respect to regular epis in \ass \cite{CarboniA:somfcr}.
\end{rem}

Consider the
doctrines on \pass
\[\Wsb{\pass}:\pass\op\ftr\ISL\qquad
\Sb{\pass}:\pass\op\ftr\ISL\qquad
\stm:\pass\op\ftr\ISL,\]
respectively the doctrine of variations, that of subobjects
and that obtained as the composite of $\Gamma$ with the contravariant
powerset functor. Clearly $\stm$ is a boolean \aritm \hyper.

\begin{lem}\label{today}
The doctrine $\Wsb{\pass}:\pass\op\ftr\ISL$ is a \whyper which
satisfies \rc and \ac on each objects and is \aritm.
\end{lem}

\begin{proof}
After \thx\ref{car}~(ii), we only need to show that
each functor
\[(\Wsb{\pass})_{\fst}:\Wsb\pass(P,T)\ftr\Wsb{\pass}((P,T)\times(M,S))\] has 
a right adjoint. Consider $[f:(Y,Z)\to(P,T)\times(M,S)]$, say that $d$
tracks $f$, let
\[Q\colon=\{\ple{x,h,t}\in P\times Y^M\times \N \mid
t\mbox{ tracks }h\mbox{ and for all }m\in M, f(h(m))=\ple{x,m}
\}\]
and let $R:\ple{x,h,t}\mapsto\epl{T(x),t}:Q\to\N$.
Define $\B_{\fst}([f])\colon=\fst:(Q,R)\to(P,T)$ which is tracked by
a code of the function \fs{(\blank)}. The function
$((x,h,t),m)\mapsto h(m):Q\times M\to Y$ is tracked by a code for the
recursive function $k\mapsto\tur{\sn{(\fs{k})}}(\sn{k})$,
thus producing an arrow $(Q,R)\times(M,S)\to (Y,Z)$ which shows that
$(\Wsb{\pass})_{\fst}(\B_{\fst}([f])\leq[f]$. The conclusion is now
straightforward.
\end{proof}

There are obvious 1-arrows of elementary doctrines 
$\stm\to\Sb{\pass}$ 
and $\Sb{\pass}\to\Wsb{\pass}$ which are
the identity on the domain of the doctrine and monotone inclusions on
the fibres.

\begin{thm}\label{carfur}
The doctrine $\Ssb{\ass}:\ass\op\ftr\ISL$ of strong subobjects on \ass
is the elementary quotient completion of the doctrine
$\stm:\pass\op\ftr\ISL$.
\end{thm}

\begin{proof}
By the universal property of the elementary quotient completion, a
1-arrow of doctrines with stable quotients as in the diagram on the
right
\[\xymatrix@C=4em@R=1em{
{\pass\op}\ar[rd]^(.4){\stm}_(.4){}="P"\ar[dd]_{G\opp}&\\
& {\ISL}.\\
{\ass\op}\ar[ru]_(.4){\Ssb{\ass}}^(.4){}="R"
&\ar"P";"R"_c^{\kern-.4ex\cdot}}
\kern6ex
\xymatrix@C=4em@R=1em{{\ct{Q}_{\stm}\op}
\ar[rd]^(.4){\Q{\stm}}_(.4){}="P"\ar[dd]_{F\opp}&\\
&{\ISL}\\
{\ass\op}\ar[ru]_(.4){\Ssb{\ass}}^(.4){}="R"&\ar"P";"R"_b^{\kern-.4ex\cdot}}
\]
is completely determined by a 1-arrow of elementary doctrines as in
the diagram on the left.
So take $G$ as the inclusion of \pass into \ass, that preserves
all finite limits. The $(P,T)$-component of the transformation $c$
takes a subset $X$ of 
$\Gamma(P,T)$ to the strong subobject $\xymatrix@1@=2ex 
{\id{P}\rst_{X}:(X,T\cap(X\times\N))\ \ar@{>->}[r]&(P,T)}$; it is an
isomorphism because of the characterization of strong subobjects in
\ass. The induced functor $F:\ct{Q}_{\stm}\ftr\ass$
is faithful. It is also full as partitioned assemblies are regular
projective (see  \rmx\ref{rcar}). Finally, $F$ is essential
surjective because \ass has enough regular projectives (see again
\rmx\ref{rcar}).
\end{proof}

\begin{cor}\label{pasmwcc} 
The doctrine $\Ssb{\ass}$ is \aritm.
\end{cor}

\begin{proof} 
It follows from \thx\ref{carfur} since $\stm$
is \aritm by \lmx\ref{today}. Thus its elementary quotient
completion is \aritm by \lmx\ref{eqc-aritm}. 
\end{proof}

\begin{prop}\label{super} 
The doctrine $\stm$ satisfies $(\wttct)$.
\end{prop}

\begin{proof}
Consider the weak exponential $(W,V)$ of $(\N,\id{N})$ to its power
as
\[W\colon=\{(g,t)\in \N^\N\times\N\mid t\mbox{ tracks }g\}\]
and $V$ the (restriction of the) second projection.
But $t$ tracks $g$ exactly when $g=\tur{t}$. So $\Gamma(W,V)$ is (in
bijection) with the set of total recursive functions on \N, and 
$\vdash \Forall f:(W,V). \ttch(f)$ in $\stm$
because
for all $(g,t)\in |W|$ there is $e\colon=t\in\N$ such that for
all $x\in\N$ there is $y\in\N$ such that
$\TU(e,x,y)=1\Land\UU(y)=f(x)$.
\end{proof}

\begin{lem}\label{super1} 
The formula $f:(W,V),e:\NN,x:\NN,y:\NN\mid\kleene{w}(e,x,y,f)$ in
$\stm(W\times\NN\times\NN\times\NN)$ 
has a \ska for $($the third
occurrence of$)$ $\NN$.
\end{lem}

\begin{proof}
Consider the function $W\times\N\times\N\to\N$ defined as
\[((g,t),e,x)\mapsto\min\{y\in\N\mid
\TU(e,x,y)=1\Land\UU(y)=g(x)\}.\]
This function is tracked by (a code of) the partial recursive function
\[(t,e,x)\mapsto \min\{y \in\N\mid
\TU(e,x,y)=1\Land\UU(y)=\tur{t}(x)\}.\]
For every element $(g,t)\in W$ the set of numbers on the right-hand
side is non-empty since $t$ belongs to it. So it defines an arrow 
$\gamma:(W,V)\times\NN\times\NN\to\NN$ in \pass which is clearly the
required \ska.
\end{proof}

\begin{lem}\label{super2} 
The formula $f:(W,V),e:\NN\mid\Forall x:\NN.\Exists
y:\NN. \kleene{w}(e,x,y,f)$ in 
$\stm(W\times\NN)$ has a \ska for $\NN$. 
\end{lem}

\begin{proof} 
A \ska $\epsilon:W\to \NN$ is determined by the function 
$(g,t)\mapsto t:W\to\N$.
\end{proof}

Applying the results in previous sections we will obtain that 
\begin{itemize}
\item $(\wttct)$ holds in \Wsb{\pass}, \Ssb{\ass} and \Sb{\eff};
\item $(\fct)$ holds in \Wsb{\pass} and \Sb{\eff};
\end{itemize}
as these are all essentially inherited from the validity of $(\wttct)$
in \stm. 

Accordingly with the previous sections, we shall write $\Pi$ and
$\Sigma$ for the universal and the existential quantification in
$\Wsb{\pass}$, while we will write $\forall$ and $\exists$ for the
universal and the existential quantification in $\stm$.

\begin{prop}\label{superweak} 
The doctrine \Wsb{\pass} satisfies $(\wttct)$. 
\end{prop}

\begin{proof}
By \prx\ref{aggiunzione} there is a right adjoint
of doctrines $\RR:\stm\ntr\Wsb{\pass}$. Let $(W,V)$ be
the weak exponential of $(\N,\id{N})$ to its power with weak
evaluation $w$. Since $\RR$ maps equality predicates to equality
predicates and commutes with substitutions, the formula $\kleene{w}$
for \Wsb{\pass} is the image under $\RR$ of 
$\kleene{w}$ for \stm. \crx\ref{ddd}, \lmx\ref{super1} and
\lmx\ref{super2} ensure that, in \Wsb{\pass}, 
\[\vdash \PI f:(W,V). \SIGMA e:\NN.\PI x:\NN.\SIGMA y:\NN. 
\cmp{\kleene{w}}(e,x,y,f).\]
if and only if
\[\vdash \cmp{\Forall f:(W,V). \Exists e:\NN.\Forall x:\NN.\Exists y:\NN. 
\kleene{w}(e,x,y,f)}.\]
And this holds by \prx\ref{super}.
\end{proof}

\begin{cor}\label{superweak2}
\begin{enumerate}\thmitem
\item The doctrines \Ssb{\ass} and \Sb{\eff} satisfy $(\ttct)$. 
\item The doctrines \Wsb{\pass} and \Sb{\eff} satisfy $(\fct)$.
\end{enumerate}
\end{cor}

\begin{proof} (i) The doctrine $\stm$ satisfies $(\ttct)$ by
\prx\ref{super}. By \thx\ref{carfur}, the doctrine $\Ssb{\ass}$ is
(equivalent to) $\Q{\stm}$, so \prx\ref{CT} applies, and
\Ssb{\ass} satisfies $(\ttct)$. Besides,
the doctrine $\Wsb{\pass}$ satisfies $(\ttct)$ by
\prx\ref{superweak}. By the results in \cite{RobinsonE:abslr}, the
topos $\eff$ is (equivalent to) $\pass\exl$; hence 
\crx\ref{completions-CT} applies, and the doctrine
\Sb{\eff} satisfies $(\ttct)$.

\noindent(ii)
By \prx\ref{choiceCT}, it suffices to show that both \Wsb{\pass} and
\Sb{\eff} satisfy $(\ttct)$ and \rcn. 
The doctrine \Wsb{\pass} satisfies $(\ttct)$ by \prx\ref{superweak};
the doctrine \Sb{\eff} 
satisfies $(\ttct)$ by \crx\ref{completions-CT}. As for \rcn, any \wsd
satisfies \ac so, in particular, it satisfies \rcn. Hence, by
\crx\ref{RCN}, \Sb{\eff} satisfies \rcn as well.
\end{proof}

\begin{rem}
As is well known, \pass is not cartesian closed, see \eg \cite{eff}.
One can see also that this is so because of the validity in
\Wsb{\pass} of $(\fct)$ and of \ac.
Indeed, if \pass were cartesian closed, since it has finite limits it
would satisfy {\it extensionality of functions} in the following
form: for all object $X,Y$ in \pass 
and $f,g:X\to Y$
\[\vdash \Forall x:A. (f(x)\eq{Y}g(x))\Implies
(\lambda x. f(x)\eq{Y^X}\lambda x. g(x))\]
where $Y^X$ indicates the exponential of $Y$ over $X$ and
$\lambda x. f(x)$ is the usual $\lambda$-notation for the abstraction
of $f$. But it is well known, see for example \cite{DT88}, 
that $(\fct)$ and \ac are inconsistent with the extensionality of function
in a many-sorted first order theory including arithmetic and finite
types. 
\end{rem}

\begin{rem}
The validity of $(\ttct)$ in \Ssb{\ass} implies that neither $(\fct)$ nor
\auc (hence \ac) are valid in \Ssb{\ass}, as its underlying logic is
boolean, see \cite{mtt} for a logical argument. 
\end{rem}

\begin{rem}\label{remq} Observe that, since the $\nno$ in $\ass$
  coincides with that in \pass, \lmx\ref{super1} 
and \lmx\ref{super2} can be proved also for $\Ssb{\ass}$. By
\rmx\ref{pippopippo} there is an adjunction between
$\Ssb{\ass}$ and $\Sb{\ass}$ that satisfies the hypotheses of
\crx\ref{ddd}. Hence \Sb{\ass} satisfies $(\ttct)$ by
\crx\ref{superweak2}-(i). We also know that $\Sb{\ass}$ satisfies
$(\fct)$, but an abstract proof of this requires an abstract treatment of
the regular completion of a lex category, which we do not include
here. We just stress that the regular completion of a lex category can
be obtained as an instance of a more general construction introduced
in \cite{MPR} that involves elementary doctrines and that produces
\ass when such a construction is performed over $\Wsb{\pass}$. 
\end{rem}

To compare $\stm$ and $\Wsb{\pass}$ we can apply the reflection
in \prx\ref{aggiunzione}. In this case it turn out that the object of
a $\stm$ over $A$ in \pass coincides with the double negated
objects of $\Wsb{\pass}$ over $A$.
This fact can be deduced from a general result.

\begin{prop}\label{notnotnot} Suppose $P$ is a \whyper with full weak
comprehensions and comprehensive diagonals. Suppose that for every
$f$ in \ct{C} the left adjoint along $f$ is stable under the double
negation, \ie $\D_f = \neg\neg\D_f$, and that $\RR$ preserves bottom
elements, \ie $[\cmp{\ff_A}]$ is the bottom element in
$\Wsb{\ct{C}}(A)$.
Then $\RR\circ \LL{}:\Wsb{\ct{C}}\to\Wsb{\ct{C}}$ coincide with the
double negation, \ie it maps \ec{f} to $\neg\neg\ec{f}$. 
\end{prop}.

\begin{proof} Consider $f:X\to A$ and recall that $\RR\circ
\LL{}([f])=[\cmp{\D_f\tt_X}]$. Note that \crx\ref{ddd} implies that
$\RR$ commutes with the universal quantification and with the
implication. Then it is 
\[[\cmp{\D_f\tt_X}]=[\cmp{\neg\neg\D_f\tt_X}]=[\cmp{\neg\B_f\neg\tt_X}]=
[\neg\Pi_f\neg\cmp{\tt_X}]=\neg\neg[\Sigma_f\cmp{\tt_X}]=
\neg\neg[\Sigma_f(\id{X})]\]
and hence the claim as $[\Sigma_f(\id{X})]=[f]$.
\end{proof}

Hyland in \cite{eff} showed that assemblies are the
$\neg\neg$-separated objects of $\eff$ for the Lawvere-Tierney
topology of double negation, \ie an object of $\eff$ is in $\ass$ if
and only if its equality predicate is $\neg\neg$-closed. This is also
a corollary of our previous results.

\begin{prop}[Hyland]\label{separated}
The category \ass is the full reflective subcategory
of \eff on $\neg\neg$-separated objects. 
\end{prop}

\begin{proof}
The 1-arrow $\RR:\stm\ntr\Wsb{\pass}$ is full
and has a left adjoint $\LL{}$ by \crx\ref{aggiunzione}. Since the
elementary quotient completion is a 2-functor, there is a full and
faithful functor $G:\ct{Q}_{\stm}\to \ct{Q}_{\Wsb{\pass}}$
which has a left adjoint $F$. Therefore $(A,\rho)$ in
$\ct{Q}_{\Wsb{\pass}}$ is in 
$\ct{Q}_{\stm}$ if and only if $(A,\rho)\simeq
GF(A,\rho)$. From the construction of $F$ and $G$ and from
\prx\ref{notnotnot} this happens if and only if 
\[(A,\rho)\simeq(A,\RR_{A\times A}\circ \LL{}_{A\times A}(\rho))=
(A,\neg\neg \rho)\]
But $\rho$ is the equality predicate over $(A,\rho)$ for the doctrine
$\Q{\Wsb{\pass}}\equiv\Sb{eff}$ by \crx\ref{car3}. The claim follows
from $\ct{Q}_{\stm}\equiv \ass$ by \thx\ref{carfur}. 
\end{proof}

\begin{rem}\label{nuovo}
The category $\ct{Equ}$ of equilogical spaces introduced in
\cite{ScottD:newcds} is the domain of the elementary quotient
completion of the doctrine of subspace inclusions on $\ct{Top}_0$, see
\cite{MPR}, and also \cite{Pasquali2018} where a more
general situation is considered. In the same vein, one can show that 
 $\ct{Equ}$ is the full and reflective subcategory of
$(\ct{Top}_0)\exl$ on those objects whose equality predicate is stable
under the double negation \cite{RosoliniG:equsfs}. 
\end{rem}

\section{Kleene's realizability interpretation in
 \Wsb{\pass}}\label{kri}

It is well known that the interpretation of Intuitionistic Arithmetic (\ha)
in the internal logic of \eff, \ie the \hyper
$\Sb{\eff}:\eff\op\ftr\ISL$, extends Kleene's realizability
interpretation, see \cite{eff,OostenJ:reaait}. 
However this is not evident in the tripos which produces
\eff as explained in \cite{HylandJ:trit} since the tripos does not 
validate Intuitionistic Arithmetic.

Here we show that \Wsb{\pass} is responsible for that result since
\eff is the domain of the elementary quotient completion of \Wsb{\pass}.
Hence \eff inherits the interpretation of connectives and quantifiers
from \Wsb{\pass}, as explained in \cite{MaiettiME:quofcm}. 

The theory \ha is interpreted in the arithmetic \whyper
$\Wsb{\pass}:\pass\op\ftr\ISL$ taking the \nno
$\NN\colon=(\N,\id{\N})$ in \pass 
as the domain of the interpretation and interpreting the operations
with the standard operations on the \nno.

For a formula $\phi$ in \ha with at most $n$ free variables
$x_1,\dots,x_n$, let $[\phi^I:X\to\N^n]$ be interpretation of $\phi$
in \Wsb{\pass} as in \cite{JacobsB:catltt}. Write instead
\[\rz{\phi}\colon=\{(k_1,\ldots,k_n,m)\in\N^{n+1}\mid
m\Vdash_\rkl\phi[k_1/x_1,\dots,k_n/x_n]\},\]
where $\Vdash_\rkl$ is Kleene realizability as presented in
\cite{DT88}, and let $\gm{\phi}:\rz{\phi}\to\N$ be the function which
maps an $(n+1)$-ple to its encoding. If we let
$\dsi{\phi}:\rz{\phi}\to\N^n$ be the projection on the first $n$
components, we obtain an arrow of partitioned assemblies
$\dsi{\phi}:(\rz{\phi},\gm{\phi})\to\NN^n$. So
\ec{\dsi{\phi}:(\rz{\phi},\gm{\phi})\to\NN^n} is an object of
$\Wsb{\pass}(\NN^n)$.

\begin{prop}\label{kleeneequiv}
For any \ha-formula $\phi$ with at most $n$ free variables
$x_1,\dots,x_n$, it is
\[\ec{\phi^I}=\ec{\dsi{\phi}}.\]
\end{prop}

The proof is an easy induction on the height of the formula $\phi$ and
it is based on the constructions in \pass. 

\begin{cor}\label{kleenearitm} A sentence $\phi$ in the language of \ha 
is true in the standard interpretation in $\Wsb{\pass}$, if and only
if $\phi$ has a realizer in the sense of Kleene realizability
interpretation in \cite{KleeneS:intint}.
\end{cor}

\begin{cor}[Hyland]\label{kleenearitm2} A sentence $\phi$ in the language of \ha 
is true in the standard interpretation in $\Sb{\eff}$, if and only
if $\phi$ has a realizer in the sense of Kleene realizability
interpretation in \cite{KleeneS:intint}.
\end{cor}

\begin{proof}
By \crx\ref{car3} $\Sb{\eff}$ is the elementary quotient completion of
$\Wsb{\pass}$, therefore the \nno in \eff is of the form
$(\NN,\delta_\NN)$ where $\NN$ is a \nno in \pass, then not only
$\Sb{\eff}(\NN,\delta_\NN)=\Wsb{\pass}(\NN)$, but quantifications,
connectives and the equality predicate are the same. The claim follows
by \crx\ref{kleenearitm}.
\end{proof}

\section{Conclusion}
We have shown that 
\eff and \ass are elementary quotient completion of suitable doctrines.
This fact is crucial to build models for the Minimalist Foundation
(\mf), introduced in \cite{mtt,m09}, extended with the various forms
of \fct. The reason is 
that \mf in \cite{m09} has a two-level structure with an extensional
level interpreted in the elementary quotient completion of its
intensional level, as analyzed categorically in
\cite{MaiettiME:quofcm}. Hence modeling \mf in \eff (or in \ass)
corresponds to build a morphism of doctrines from the 
elementary quotient completion of the syntactic doctrine of \mf to doctrines based on \eff (or \ass).

In particular we would like to embed in \eff
the already known models which provides the consistency of both levels
of \mf with \fct in \cite{types15,trip,IMMSaml}. Since these models
provide extraction of programs from constructive proofs in \mf as
shown in \cite{tamc}, we think that \eff should provide a framework to
extend extraction of programs from proofs to extensions of \mf with
general inductive definitions.

Finally we would also like to exploit the categorical structure
of \ass to build models similar to that in \cite{StreicherT:indipa}
in order to show consistency of \mf (and of its extensions with
inductive definitions) with classical logic and the weak form of \fct
valid in the doctrine $\pw\Gamma$.

\end{document}